\documentclass[12pt]{article}

\usepackage[utf8]{inputenc}

\usepackage[english]{babel}

\usepackage{amsmath,amsfonts,amssymb,amsthm}

\usepackage{amscd}

\usepackage[a4paper,nomarginpar]{geometry}

\usepackage{epsfig, graphicx, color}

\usepackage[all]{xy}

\newtheorem{theorem}{Theorem}
\newtheorem{lemma}[theorem]{Lemma}

\newtheorem{corollary}[theorem]{Corollary}

\theoremstyle{definition}

\theoremstyle{remark}
\newtheorem{remark}[theorem]{Remark}

\newcommand{\N}{\mathbb{N}}  % set of natural numbers
\newcommand{\Z}{\mathbb{Z}}  % set of integer numbers
\newcommand{\eps}{\varepsilon} % abbreviation for epsilon
\newcommand{\ov} {\overline}   % abbreviation for overline
\newcommand{\wt} {\widetilde}  % abbreviation for widetilde

\newcommand{\cA}{\mathcal{A}}
\newcommand{\cB}{\mathcal{B}}
\newcommand{\cC}{\mathcal{C}}
\newcommand{\cD}{\mathcal{D}}
\newcommand{\cH}{\mathcal{H}}
\newcommand{\cK}{\mathcal{K}}
\newcommand{\cM}{\mathcal{M}}

\newcommand{\p} {\mathcal{P}}
\newcommand{\q} {\mathcal{Q}}
\newcommand{\cS}{\mathcal{S}}

\newcommand{\udens}{\operatorname{\overline{dens}}}

\newcommand{\gr}  {\operatorname{Gr}}
\newcommand{\mesh}{\operatorname{mesh}}
\newcommand{\diam}{\operatorname{diam}}
\newcommand{\ent} {\operatorname{ent}}
\newcommand{\card}{\operatorname{card}}

\newcommand{\cantor}    {{\{0,1\}^\N}}
\newcommand{\contcantor}{\cC(\cantor)}
\newcommand{\homcantor} {\cH(\cantor)}

\newcommand{\mcantor}   {\cM(\cantor)}

\newcommand{\edge}[1]{\ensuremath{\overrightarrow  {#1 }}} % edge notation

\begin{document}

\title{On the Dynamics of Induced Maps\\ on the Space of Probability Measures}

\author{Nilson C. Bernardes Jr. \ and \ R\^omulo M. Vermersch}

\date{ }

\maketitle

\begin{abstract}
For the generic continuous map and for the generic homeomorphism of the
Cantor space, we study the dynamics of the induced map on the space of
probability measures, with emphasis on the notions of Li-Yorke chaos,
topological entropy, equicontinuity, chain continuity, chain mixing,
shadowing and recurrence. We also establish some results concerning
induced maps that hold on arbitrary compact metric spaces.
\end{abstract}

%%%%%%%%%%%%%%%%%%%%%%%%%%%%%%%%%%%%%%%%%%%%%%%%%%%%%%%%%%%%%%%%%%%%%%%%%%%%%

\section{Introduction}

\hspace*{4mm} Let $M$ be a compact metric space with metric $d$ and let
$\cB_M$ be the set of all Borel subsets of $M$.
We denote by $\cC(M)$ (resp.\ $\cH(M)$) the space of all continuous maps
from $M$ into $M$ (resp.\ of all homeomorphisms from $M$ onto $M$)
endowed with the metric
$$
\tilde{d}(f,g):= \max_{x \in M} d(f(x),g(x)).
$$
Moreover, we denote by $\cK(M)$ the hyperspace of all nonempty closed subsets
of $M$ endowed with the Hausdorff metric
$$
d_H(X,Y):= \max \big\{\max_{x \in X} d(x,Y),\max_{y \in Y} d(y,X) \big\},
$$
and by $\cM(M)$ the space of all Borel probability measures on $M$
endowed with the Prohorov metric
$$
d_P(\mu,\nu):= \inf\{\delta > 0 : \mu(X) \leq \nu(X^\delta) + \delta
               \text{ and } \nu(X) \leq \mu(X^\delta) + \delta
               \text{ for all } X \in \cB_M\},
$$
where $X^\delta:= \{x \in M : d(x,X) < \delta\}$ is the $\delta$-neighborhood
of $X$ ($X \subset M$). The Prohorov metric induces the usual weak topology
for measures. It is well known that both $\cK(M)$ and $\cM(M)$ are compact
metric spaces. Moreover,
$$
d_P(\mu,\nu) = \inf\{\delta > 0 : \mu(X) \leq \nu(X^\delta) + \delta
               \text{ for all } X \in \cB_M\}
$$
(\cite{PBil99}, Page 72).
Given $f \in \cC(M)$, the induced maps $\ov{f} : \cK(M) \to \cK(M)$ and
$\wt{f} : \cM(M) \to \cM(M)$ are the continuous maps given by
$$
\ov{f}(X):= f(X) \ \ \ \ \ (X \in \cK(M))
$$
and
$$
(\wt{f}(\mu))(X):= \mu(f^{-1}(X)) \ \ \ \ \ (\mu \in \cM(M), X \in \cB_M).
$$
If $f$ is a homeomorphism, then so are $\ov{f}$ and $\wt{f}$.
We refer the reader to the books \cite{AIllSNad99} and \cite{PBil99}
for a study of the spaces $\cK(M)$ and $\cM(M)$, respectively.

\smallskip
Given a Baire space $Z$, to say that ``the generic element of $Z$ has a
certain property $P$'' means that the set of all elements of $Z$ that
do not satisfy property $P$ is of the first category in $Z$.
The word ``typical'' is sometimes used instead of the word ``generic''.

\smallskip
A systematic study of the dynamics of the induced maps $\ov{f}$ and $\wt{f}$
was initiated by Bauer and Sigmund \cite{WBauKSig75} and has been developed
by several authors; see \cite{AcoIllMen09, EGlaBWei95, GuiKwiLamOprPer09,
KSig78}, for instance. 
On the other hand, the study of generic dynamics is a classical topic
in the area of dynamical systems. In the context of topological dynamics,
such a study has been developed by several authors during the last forty
years. We refer the reader to \cite{AgrBruLac89, AkiHurKen03, NBer08, NBer13,
NBerUDar12, SPil94}, where further references can be found.

\smallskip
In \cite{NBerRVer14} the authors combined both topics and developed
a detailed study of the dynamics of the induced map $\ov{f}$ in the case
$f$ is the generic continuous map or the generic homeomorphism of the
Cantor space. Our main goal in the present paper is to develop such a
detailed study for the induced map $\wt{f}$. It turns out that in many
aspects the dynamics of the induced map $\wt{f}$ is completely different from
the dynamics of the induced map $\ov{f}$. For instance, for the generic
homeomorphism $f$ of the Cantor space, it was proved in \cite{NBerRVer14}
that the induced map $\ov{f}$ is uniformly distributionally chaotic,
has infinite topological entropy, has the shadowing property and is
chain continuous at every point of a dense open set, but we will see that
the induced map $\wt{f}$ has no Li-Yorke pair, has zero topological entropy,
does not have the shadowing property and is chain continuous at no point
at all.

\smallskip
We also consider the problem of the density of the set of periodic points
in the set of nonwandering points. The {\em $C^1$ closing lemma} and
the associated {\em $C^1$ general density theorem} are fundamental results
in the theory of smooth dynamical systems due to Pugh~\cite{CPug67}.
The former says that if $x$ is a nonwandering point of a diffeomorphism
$f$ on a compact smooth manifold $M$, then in any $C^1$ neighborhood
of $f$ there is a diffeomorphism $g$ for which $x$ is a periodic point.
The latter says that $C^1$ generically the set of periodic points of a
diffeomorphism is dense in the set of its nonwandering points.
The corresponding {\em $C^0$ general density theorem} for homeomorphisms
on a compact smooth manifold $M$ was announced by Palis, Pugh, Shub and
Sullivan \cite{PalPugShuSul75}, but a flaw in their argument was later
described by Coven, Madden and Nitecki \cite{CovMadNit82}, who proposed
a different argument. However, it was pointed out by Pilyugin \cite{SPil94}
that the argument in \cite{CovMadNit82} only works in the $C^0$ closure
of the set of diffeomorphisms in the set of homeomorphisms on $M$.
By Munkres \cite{JMun60} and
Whitehead \cite{JWhi61}, this $C^0$ closure is equal to the set of all
homeomorphisms on $M$ whenever $\dim M \leq 3$, but Munkres \cite{JMun60}
also showed that this is not the case if $\dim M > 3$. A proof of the full
$C^0$ general density theorem was finally given by Hurley \cite{MHur96},
who also observed how to adapt the argument to obtain the corresponding
result for continuous maps. In the case of the Cantor space the situation
is completely different: the generic continuous map and the generic
homeomorphism of the Cantor space have no periodic point! This was observed
by D'Aniello and Darji \cite{EDAnUDar11} and by Akin, Hurley and Kennedy
\cite{AkiHurKen03}, respectively. Nevertheless, surprisingly enough,
we will see that for the generic continuous map (resp.\ the generic
homeomorphism) $f$ of the Cantor space, the set of periodic points
of the induced map $\wt{f}$ is dense in the set of its nonwandering points.
Moreover, several additional properties of the induced map $\wt{f}$ related
to recurrence will also be established.

\smallskip
Although our main motivation was to study the dynamics of the induced map
$\wt{f}$ in the case $f$ is the generic map of the Cantor space, in doing so
we have also established some results that hold on arbitrary compact metric
spaces. For instance, we will see that for any homeomorphism $f$ of any
compact metric space $M$, the following properties hold:
\begin{itemize}
\item $\wt{f}$ is chain mixing.
\item $\wt{f}$ has no point of chain continuity (provided $M$ is not a
      singleton).
\end{itemize}
Moreover, for any continuous map $f$ of any compact metric space $M$,
the following assertions are equivalent:
\begin{description}
\item {(i)}   $\wt{f}$ is chain continuous at some point;
\item {(ii)}  $\wt{f}$ is chain continuous at every point;
\item {(iii)} $\bigcap_{n=1}^\infty f^n(M)$ is a singleton.
\end{description}
Such results complement the previous works of
Bauer and Sigmund \cite{WBauKSig75}, Sigmund \cite{KSig78} and
Glasner and Weiss \cite{EGlaBWei95} on the dynamics of the induced map
$\wt{f}$ in the context of general compact metric spaces.

%%%%%%%%%%%%%%%%%%%%%%%%%%%%%%%%%%%%%%%%%%%%%%%%%%%%%%%%%%%%%%%%%%%%%%%%%%%%%

\section{Preliminaries}

\hspace*{4mm} Throughout $M$ denotes an arbitrary compact metric space
with metric $d$. For each $x \in M$ and each $r > 0$,
$B(x;r):= \{y \in M : d(y,x) < r\}$ is the open ball with center
$x$ and radius $r$.

\smallskip
Our model for the Cantor space is the product space $\cantor$,
where $\{0,1\}$ is endowed with the discrete topology. We consider
$\cantor$ endowed with the compatible metric, also denoted by $d$, given by
$d(\sigma,\sigma):= 0$ and $d(\sigma,\tau):= \frac{1}{n}$ where $n$
is the least positive integer such that $\sigma(n) \neq \tau(n)$
($\sigma, \tau \in \cantor$, $\sigma \neq \tau$).

\smallskip
The main tools used in the present paper for the results concerning the
generic continuous map and the generic homeomorphism of the Cantor space are
the graph theoretic descriptions of these maps obtained in \cite{NBerUDar12}.
For the sake of completeness, let us briefly recall these descriptions. 

\smallskip
A {\em partition} of $\cantor$ is a finite collection $\p$ of pairwise
disjoint nonempty clopen sets whose union is $\cantor$, and $\mesh(\p)$
is the maximum diameter of the elements of $\p$.
For each $f \in \contcantor$ and each partition $\p$ of $\cantor$,
we consider the digraph $\gr(f,\p)$ whose vertex set is $\p$ and
whose edge set is
$$
\{\edge{ab} : a,b \in \p \text{ and } f(a) \cap b \neq \emptyset\}.
$$

A {\em component} of a digraph $G$ is a largest (in vertices and edges)
subgraph $H$ of $G$ such that given any two vertices $a,b$ in $H$,
there are vertices $a_1,\ldots,a_n$ in $H$ such that $a_1 = a$, $a_n = b$
and, for each $1 \leq i < n$, $\edge{a_ia_{i+1}}$ or $\edge{a_{i+1}a_i}$
is an edge of $H$.

\smallskip
A digraph $\ell$ is a {\em loop of length $n$} if the vertex set of $\ell$
is a set $\{v_1,\ldots,v_n\}$ with $n$ elements and the edges of $\ell$
are $\edge{v_nv_1}$ and $\edge{v_iv_{i+1}}$ for $1 \leq i < n$.

\smallskip
A digraph $B$ is a {\em balloon of type $(s,t)$} if the vertex set of $B$
is the union of two disjoint sets $p = \{v_1,\ldots,v_s\}$ and
$\ell = \{w_1,\ldots,w_t\}$, and the edges of $B$ are
the edges of the path $p$ (i.e., $\edge{v_iv_{i+1}}$ for $1 \leq i < s$),
the edges of the loop formed by $\ell$, and $\edge{v_sw_1}$.
We call $v_1$ the {\em initial vertex} of $B$.

\smallskip
A digraph $D$ is a {\em dumbbell of type $(r,s,t)$} if the vertex set of $D$
is the union of three disjoint sets $\ell_1 = \{u_1,\ldots,u_r\}$,
$p = \{v_1,\ldots,v_s\}$ and $\ell_2 = \{w_1,\ldots,w_t\}$, and the edges of
$D$ are the edges of the loops formed by $\ell_1$ and $\ell_2$,
the edges of the path $p$, $\edge{u_1v_1}$ and $\edge{v_sw_1}$.
If $r = t$ then we say that the dumbbell is {\em balanced with plate weight
$r$}.

\smallskip
Suppose that $f \in \contcantor$, $\p$ is a partition of $\cantor$ and
$B$ is a component of $\gr(f,\p)$ which is a balloon.
Write
$$
B = \{v_1,\ldots,v_s\} \cup \{w_1,\ldots,w_t\},
$$
with usual labeling. We say that the balloon $B$ is {\em strict relative to
$f$} if $f(v_i) \subsetneq v_{i+1}$ for every $1 \leq i < s$,
$f(w_j) \subsetneq w_{j+1}$ for every $1 \leq j < t$, and
$f(v_s) \cup f(w_t) \subsetneq w_1$.

\smallskip
Suppose that $h \in \homcantor$, $\p$ is a partition of $\cantor$ and
$D$ is a component of $\gr(h,\p)$ which is a dumbbell.
Write
$$
D = \{u_1,\ldots,u_r\} \cup \{v_1,\ldots,v_s\} \cup \{w_1,\ldots,w_t\},
$$
with usual labeling. We say that the dumbbell $D$ {\em contains a left loop
of $h$} (resp.\ {\em a right loop of $h$}) if there is a nonempty clopen
subset $a$ of $u_1$ (resp.\ of $w_1$) such that $h^r(a) = a$
(resp.\ $h^t(a) = a$).

\smallskip
Let us now recall the above-mentioned results from \cite{NBerUDar12}:

\vspace{.1in}
\noindent {\bf Theorem A.} {\it The generic $f \in \contcantor$ has the
following property:

\smallskip
\noindent {\rm (Q)} For every $m \in \N$, there are a partition $\p$
of $\cantor$ of mesh $< 1/m$ and a multiple $q \in \N$ of $m$ such that
every component of $\gr(f,\p)$ is a balloon of type $(q!,q!)$
which is strict relative to $f$.}

\vspace{.1in}
\noindent {\bf Theorem B.} {\it The generic $h \in \homcantor$ has the
following property:

\smallskip
\noindent {\rm (P)} For every $m \in \N$, there are a partition $\p$
of $\cantor$ of mesh $< 1/m$ and a multiple $q \in \N$ of $m$ such that
every component of $\gr(h,\p)$ is a balanced dumbbell with plate weight $q!$
that contains both a left and a right loop of $h$.}

\vspace{.1in}
Moreover, it was proved in \cite{NBerUDar12} that any two maps
$f, g \in \contcantor$ (resp.\ $f, g \in \homcantor$)
with property (Q) (resp.\ property (P)) are topologically conjugate
to each other, that is, $f = h^{-1} g h$ for some $h \in \homcantor$.

\smallskip
Given a partition $\p$ of $\cantor$, we define
$$
\delta(\p):= \min\{d(a,b) : a,b \in \p, a \neq b\} > 0
$$
and
$$
I_\p(X):= \{a \in \p : a \cap X \neq \emptyset\} \ \ \ (X \subset \cantor).
$$

\smallskip
For each $z \in M$, $\pi_z \in \cM(M)$ denotes the unit mass concentrated
at $z$. Note that
$$
d_P(\pi_z,\pi_w) = \min\{d(z,w),1\}.
$$
Moreover, for every $f \in \cC(M)$,
$$
\wt{f}(\pi_z) = \pi_{f(z)}.
$$

%%%%%%%%%%%%%%%%%%%%%%%%%%%%%%%%%%%%%%%%%%%%%%%%%%%%%%%%%%%%%%%%%%%%%%%%%%%%%

\section{Li-Yorke chaos and topological entropy}

\hspace*{4mm}
Let us begin by recalling the notions of Li-Yorke chaos \cite{TLiJYor75}
and distributional chaos \cite{BSchJSmi94}.
Given $f \in \cC(M)$, recall that $(x,y)$ is a {\em Li-Yorke pair} for $f$ if
$$
\liminf_{n \to \infty} d(f^n(x),f^n(y)) = 0 \ \ \ \text{ and } \ \ \
\limsup_{n \to \infty} d(f^n(x),f^n(y)) > 0.
$$
The map $f$ is {\em Li-Yorke chaotic} if there is an uncountable set
$S$ (a {\em scrambled set} for $f$) such that $(x,y)$ is a Li-Yorke pair
for $f$ whenever $x$ and $y$ are distinct points in~$S$. Moreover,
$(x,y)$ is a {\em distributionally $\eps$-chaotic pair} for $f$ ($\eps > 0$)
if
$$
\udens\{n \in \N : d(f^n(x),f^n(y)) \geq \eps\} = 1
$$
and
$$
\udens\{n \in \N : d(f^n(x),f^n(y)) < \delta\} = 1,
$$
for all $\delta > 0$, where
$$
\udens(A):= \limsup_{n \to \infty} \frac{\card([1,n] \cap A)}{n}
$$
is the {\em upper density} of the subset $A$ of $\N$.
The map $f$ is {\em uniformly distributionally chaotic} if there is an
uncountable set $S$ (a {\em distributionally $\eps$-scrambled set} for $f$)
such that $(x,y)$ is a distributionally $\eps$-chaotic pair for $f$ whenever
$x$ and $y$ are distinct points in $S$.

\smallskip
It was proved in \cite{NBerRVer14} that for the generic $h \in \homcantor$,
the induced map $\ov{h}$ is uniformly distributionally chaotic. Surprisingly
enough, we shall now see that the induced map $\wt{h}$ is not even Li-Yorke
chaotic. In fact, even the following stronger statement holds.

\begin{theorem}\label{LYHom}
For the generic $h \in \homcantor$, $\wt{h}$ has no Li-Yorke pair.
\end{theorem}

\begin{proof}
Let $h \in \homcantor$ satisfy property (P) of Theorem~B. Let $(\mu,\nu)$ be
a pair of elements of $\mcantor$ and suppose that
$$
\limsup_{n \to \infty} d_P(\wt{h}^{\,n}(\mu),\wt{h}^{\,n}(\nu)) > 0.
$$
Then, we may fix $\eps > 0$ such that
\begin{equation}
d_P(\wt{h}^{\,n}(\mu),\wt{h}^{\,n}(\nu)) > \eps \label{eq1}
\end{equation}
for infinitely many values of $n$. For each such $n$, there is a Borel
subset $Y_n$ of $\cantor$ such that
\begin{equation}
\mu(h^{-n}(Y_n)) > \nu(h^{-n}((Y_n)^\eps)) + \eps. \label{eq2}
\end{equation}

\smallskip
Let $\p$ be a partition of $\cantor$ of mesh $< \eps$ such that every
component of $\gr(h,\p)$ is a balanced dumbbell with plate weight $q \geq 2$.
Let
$$
D_i := \{u_{i,1},\ldots,u_{i,q}\} \cup \{v_{i,1},\ldots,v_{i,s_i}\} \cup
       \{w_{i,1},\ldots,w_{i,q}\} \ \ \ (1 \leq i \leq N)
$$
be the components (dumbbells) of $\gr(h,\p)$.
For each $i \in \{1,\ldots,N\}$, we consider the nonempty closed set
$$
X_i := F_i \cup h(F_i) \cup \ldots \cup h^{q-1}(F_i),
$$
where $F_i := \bigcap_{n=0}^\infty h^{-nq}(u_{i,1})$. Note that
$h(X_i) = X_i$, because $h^q(F_i) = F_i$. Moreover,
$(u_{i,1} \cup \ldots \cup u_{i,q}) \backslash X_i$ is exactly the set
of all $\sigma \in u_{i,1} \cup \ldots \cup u_{i,q}$ whose forward
trajectory eventually goes to the bar of the dumbbell $D_i$, that is,
$h^r(\sigma) \in v_{i,1}$ for some $r \in \N$.

\smallskip
For each $n$ such that (\ref{eq1}) holds, we define
$A_n:= \bigcup\{a : a \in I_\p(Y_n)\} \supset Y_n$.
Since $\mesh(\p) < \eps$, $A_n \subset (Y_n)^\eps$. Thus, by (\ref{eq2}),
$$
\mu(h^{-n}(A_n)) > \nu(h^{-n}(A_n)) + \eps.
$$
Since this last inequality holds for infinitely many values of $n$
and there are only finitely many possible $A_n$'s, we see that there is
a set $A$ which is a union of some elements of $\p$ such that
\begin{equation}
\mu(h^{-n}(A)) > \nu(h^{-n}(A)) + \eps \label{eq3}
\end{equation}
for infinitely many values of $n$.

\smallskip
Since
$$
\lim_{k \to \infty} \varphi\Big( \bigcup_{i=1}^N \bigcup_{n=k}^\infty
  h^{-n}(v_{i,1})\Big) = 0
$$
for every $\varphi \in \mcantor$, we may fix $k \in \N$ such that
\begin{equation}
\mu(Z) < \eps/3 \ \ \ \ \text{ and } \ \ \ \ \nu(Z) < \eps/3, \label{eq4}
\end{equation}
where
$$
Z:= \bigcup_{i=1}^N \bigcup_{n=k}^\infty h^{-n}(v_{i,1}).
$$

Now, we decompose the set $A$ into three disjoint sets:
$$
A = B \cup C \cup D,
$$
where
\begin{align*}
B &\subset X:= \bigcup_{i=1}^N X_i,\\
C &\subset U:= \bigcup_{i=1}^N \big[((u_{i,1} \cup \ldots \cup u_{i,q})
               \backslash X_i) \cup
               (v_{i,1} \cup \ldots \cup v_{i,s_i})\big],\\
D &\subset W:= \bigcup_{i=1}^N (w_{i,1} \cup \ldots \cup w_{i,q}).
\end{align*}
Since $h^{-q}(B) = B$, $h^{-q}(D) \cap W = D$ and
$h^{-n}(C) \subset h^{-n}(U) \subset Z$ whenever $n$ is big enough,
it follows that there exists $n_0 \in \N$ such that the sequence
$$
(h^{-n}(A) \backslash Z)_{n \geq n_0}
$$
is periodic and its period divides $q$.
By (\ref{eq3}), there exists $m_0 \geq n_0$ such that
$$
\mu(h^{-m_0}(A)) > \nu(h^{-m_0}(A)) + \eps.
$$
Hence, by (\ref{eq4}),
$$
\mu(h^{-m_0}(A) \backslash Z) > \nu(h^{-m_0}(A) \backslash Z)
  + \frac{2\eps}{3}\cdot
$$
Since the period of the periodic sequence
$(h^{-n}(A) \backslash Z)_{n \geq n_0}$ divides $q$, we obtain
\begin{equation}
\mu(h^{-m_0-nq}(A) \backslash Z) > \nu(h^{-m_0-nq}(A) \backslash Z)
  + \frac{2\eps}{3} \ \ \text{ for all } n \in \N_0. \label{eq5}
\end{equation}
Choose $0 < \delta < \min\{\delta(\p),\eps/3\}$.
Then, for every $n \in \N_0$,
\begin{align*}
\mu(h^{-m_0-nq}(A))
  &\geq \mu(h^{-m_0-nq}(A) \backslash Z)\\
  &>    \nu(h^{-m_0-nq}(A) \backslash Z) + \frac{2\eps}{3}\\
  &>    \nu(h^{-m_0-nq}(A)) + \frac{\eps}{3}\\
  &>    \nu(h^{-m_0-nq}(A)) + \delta\\
  &=    \nu(h^{-m_0-nq}(A^\delta)) + \delta,
\end{align*}
where we used (\ref{eq5}), (\ref{eq4}) and the fact that $A^\delta = A$
(because $\delta < \delta(\p)$). Thus,
$$
d_P(\wt{h}^{\,m_0+nq}(\mu),\wt{h}^{\,m_0+nq}(\nu)) \geq \delta \ \
  \text{ for all } n \in \N_0.
$$
This implies that
$$
\liminf_{n \to \infty} d_P(\wt{h}^{\,n}(\mu),\wt{h}^{\,n}(\nu)) > 0,
$$
and so $(\mu,\nu)$ is not a Li-Yorke pair for $\wt{h}$.
\end{proof}

Let us now recall the notion of topological entropy. Fix $f \in \cC(M)$.
For each $n \in \N$, consider the equivalent metric $d_n$ on $M$ given by
$$
d_n(x,y):= \max_{0 \leq k < n} d(f^k(x),f^k(y)).
$$
A subset $A$ of $M$ is {\it $(n,\epsilon,f)$-separated} if
$d_n(x,y) \geq \epsilon$ for every $x, y \in A$ with $x \neq y$.
Let $N(n,\epsilon,f)$ be the maximum cardinality of an
$(n,\epsilon,f)$-separated set. The {\em topological entropy}
of $f$ is defined by
$$
\ent(f):= \lim_{\epsilon \to 0^+} \Big( \limsup_{n \to \infty}
          \frac{1}{n} \log N(n,\epsilon,f) \Big).
$$
This notion was introduced by Adler, Konheim and McAndrew \cite{AdlKonMcA65}.
Here we are adopting the equivalent definition formulated by
Bowen \cite{RBow71} and Dinaburg \cite{EDin71}.

\smallskip
Glasner and Weiss \cite{EGlaBWei95} discovered the surprising fact that
$\ent(h) = 0$ implies $\ent(\wt{h}) = 0$ ($h \in \cH(M)$). Moreover,
they proved in \cite{EGlaBWei01} that the generic $h \in \homcantor$
has topological entropy zero. Hence, by combining these two facts,
we have the following result:
\begin{quote}
{\it For the generic $h \in \homcantor$, $\ent(\wt{h}) = 0$.}
\end{quote}
We remark that Theorem~\ref{LYHom} also implies this result of Glasner
and Weiss, since homeomorphisms with positive topological entropy are
Li-Yorke chaotic \cite{BlaGlaKolMaa02}. In strong contrast, it was proved
in \cite{NBerRVer14} that for the generic $h \in \homcantor$,
$\ent(\ov{h}) = \infty$.

\smallskip
Let us now consider the case of continuous maps. For this purpose,
we will need the following lemmas.

\begin{lemma}\label{Lemma1}
Let $\p$ be a partition of $\cantor$. For every $\mu, \nu \in \mcantor$, if
$$
d_P(\mu,\nu) < \delta \leq \delta(\p),
$$
then
$$
|\mu(a) - \nu(a)| < \delta \ \text{ for all } a \in \p.
$$
\end{lemma}

\begin{proof}
Let $\gamma$ be such that $d_P(\mu,\nu) < \gamma < \delta$.
Since $\gamma < \delta(\p)$, $a^\gamma = a$ for every $a \in \p$.
Therefore, since $d_P(\mu,\nu) < \gamma$,
$$
\mu(a) \leq \nu(a^\gamma) + \gamma = \nu(a) + \gamma \ \ \ \text{ and } \ \ \
\nu(a) \leq \mu(a^\gamma) + \gamma = \mu(a) + \gamma,
$$
and so $|\mu(a) - \nu(a)| \leq \gamma < \delta$ ($a \in \p$).
\end{proof}

\begin{lemma}\label{Lemma2}
Let $\p$ be a partition of $\cantor$. For every $\mu, \nu \in \mcantor$, if
$$
|\mu(a) - \nu(a)| \leq \frac{\mesh(\p)}{\card(\p)} \ \text{ for all } a \in \p,
$$
then
$$
d_P(\mu,\nu) \leq \mesh(\p).
$$
\end{lemma}

\begin{proof}
Fix $\gamma > \mesh(\p)$. For each Borel subset $X$ of $\cantor$,
$\bigcup\{a : a \in I_\p(X)\} \subset X^\gamma$,
and so
$$
\mu(X) = \sum_{a \in I_\p(X)} \mu(X \cap a)
     \leq \sum_{a \in I_\p(X)} \mu(a)
     \leq \Big( \sum_{a \in I_\p(X)} \nu(a) \Big) + \mesh(\p)
       < \nu(X^\gamma) + \gamma.
$$
Thus, $d_P(\mu,\nu) \leq \gamma$. Since $\gamma > \mesh(\p)$ is arbitrary,
we have the desired inequality.
\end{proof}

Given $f : M \to M$, recall that $f$ is {\em equicontinuous} at a point
$x \in M$ if for every $\eps > 0$ there exists $\delta > 0$ such that
$$
d(y,x) < \delta \ \Longrightarrow \ d(f^n(y),f^n(x)) < \eps
  \text{ for all } n \geq 0.
$$

\begin{theorem}\label{Equicontinuous}
For the generic $f \in \contcantor$, $\wt{f}$ is equicontinuous
at every point.
\end{theorem}

\begin{proof}
Let $f \in \contcantor$ satisfy property (Q) of Theorem A.
Given $\eps > 0$, there exist a partition $\p$ of $\cantor$ of mesh $< \eps$
and an integer $q \geq 1$ such that every component of $\gr(f,\p)$ is a
balloon of type $(q,q)$. Let $\mu,\nu \in \mcantor$ be such that
$$
d_P(\mu,\nu) < \min\Big\{ \delta(\p),\frac{\mesh(\p)}{2 \card(\p)} \Big\}.
$$
By Lemma~\ref{Lemma1},
$$
|\mu(a) - \nu(a)| < \frac{\mesh(\p)}{2 \card(\p)} \ \ \ \ (a \in \p).
$$
Fix $a \in \p$ and $n \geq 0$. Let $B$ be the component (balloon) of
$\gr(f,\p)$ that contains $a$ as a vertex. Then,
$$
|(\wt{f}^{\,n}(\mu))(a) - (\wt{f}^{\,n}(\nu))(a)|
  = |\mu(f^{-n}(a)) - \nu(f^{-n}(a))|
  < \frac{\mesh(\p)}{\card(\p)},
$$
because $f^{-n}(a)$ is empty or a vertex of $B$ or the union of two vertices
of $B$. Therefore, it follows from Lemma~\ref{Lemma2} that
$$
d_P(\wt{f}^{\,n}(\mu),\wt{f}^{\,n}(\nu)) < \eps \ \text{ for all } n \geq 0.
$$
This completes the proof.
\end{proof}

The above theorem is not true in the case of homeomorphisms. Indeed,
it was proved in \cite{NBerUDar12} that the generic $h \in \homcantor$
is not equicontinuous at each point of an uncountable set, and so the
same is true for the induced map $\wt{h}$.

\smallskip
The above theorem has the following interesting consequences.

\begin{corollary}
For the generic $f \in \contcantor$, $\wt{f}$ has no Li-Yorke pair.
\end{corollary}

\begin{corollary}
For the generic $f \in \contcantor$, $\ent(\wt{f}) = 0$.
\end{corollary}

%%%%%%%%%%%%%%%%%%%%%%%%%%%%%%%%%%%%%%%%%%%%%%%%%%%%%%%%%%%%%%%%%%%%%%%%%%%%%

\section{Chain continuity and shadowing}

\hspace*{4mm} Given $f : M \to M$, recall that $f$ is {\em chain continuous}
at a point $x \in M$ \cite{EAki96,NBer02} if for every $\eps > 0$
there exists $\delta > 0$ such that for any choice of points
$$
x_0 \in B(x;\delta), \
x_1 \in B(f(x_0);\delta), \
x_2 \in B(f(x_1);\delta),\ldots,
$$
we have that
$$
d(x_n,f^n(x)) < \eps \ \ \text{ for all } n \geq 0.
$$
Of course, chain continuity is a much stronger property than equicontinuity.

\smallskip
It was proved in \cite{NBerRVer14} that for the generic $f \in \contcantor$
(resp.\ $h \in \homcantor$), the induced map $\ov{f}$ (resp.\ $\ov{h}$) is
chain continuous at every point (resp.\ is chain continuous at every point
of a dense open set). We shall see that the situation is completely different
for the induced map $\wt{f}$ (resp.\ $\wt{h}$). Indeed, for the generic
$f \in \contcantor$ (resp.\ $h \in \homcantor$), the induced map $\wt{f}$
(resp.\ $\wt{h}$) has no point of chain continuity. We shall obtain these
facts from more general results for arbitrary compact metric spaces.

\begin{lemma}\label{Lemma4}
If $\mu,\nu \in \cM(M)$, $\delta > 0$, $n \in \N_0$ and $1 - (n+1)\delta > 0$,
then
$$
d_P\big((1-(n+1)\delta)\mu + (n+1)\delta\nu,(1-n\delta)\mu + n\delta\nu\big)
  \leq \delta.
$$
\end{lemma}

\begin{proof}
For each Borel subset $X$ of $M$,
\begin{align*}
((1 - (n+1)\delta)\mu + (n+1)\delta\nu)(X)
  &= ((1 - n\delta)\mu + n\delta\nu)(X) + \delta(\nu(X) - \mu(X))\\
  &\leq ((1 - n\delta)\mu + n\delta\nu)(X^\delta) + \delta.
\end{align*}
This implies the desired inequality.
\end{proof}

\begin{theorem}\label{ChainC}
Given $0 < \delta < 1$, there exists $k_0 \in \N$ such that
for any $f \in \cC(M)$, any $\mu,\nu \in \cM(M)$ and any $k \geq k_0$,
there exist
$$
\mu_1 \in B(\wt{f}(\mu);\delta), \
\mu_2 \in B(\wt{f}(\mu_1);\delta),\ldots,
\mu_k \in B(\wt{f}(\mu_{k-1});\delta)
$$
such that
$$
\mu_k = \wt{f}^{\, k}(\nu).
$$
In particular, if we choose $\nu = \pi_z$ for some $z \in M$, then
$$
\mu_k = \pi_{f^k(z)}.
$$
\end{theorem}

\begin{proof}
Let $0 < \gamma < \delta$ and let $k_0 \in \N$ be such that
$(k_0 - 1)\gamma < 1 \leq k_0\gamma$. Define
$$
\mu_1:= (1 - \gamma) \wt{f}(\mu) + \gamma \wt{f}(\nu).
$$
By Lemma~\ref{Lemma4}, $d_P(\mu_1,\wt{f}(\mu)) \leq \gamma < \delta$.
Moreover,
$$
\wt{f}(\mu_1) = (1 - \gamma) \wt{f}^{\, 2}(\mu) + \gamma \wt{f}^{\, 2}(\nu).
$$
Define
$$
\mu_2:= (1 - 2 \gamma) \wt{f}^{\, 2}(\mu) + 2 \gamma \wt{f}^{\, 2}(\nu).
$$
By Lemma~\ref{Lemma4}, $d_P(\mu_2,\wt{f}(\mu_1)) \leq \gamma < \delta$.
Moreover,
$$
\wt{f}(\mu_2) = (1 - 2\gamma) \wt{f}^{\, 3}(\mu) + 2\gamma \wt{f}^{\, 3}(\nu).
$$
We continue this process until we define
$$
\mu_{k_0-1}:= (1 - (k_0 - 1)\gamma) \wt{f}^{\, k_0-1}(\mu) +
              (k_0 - 1)\gamma \wt{f}^{\, k_0-1}(\nu).
$$
Then,
$$
\wt{f}(\mu_{k_0-1}) = (1 - (k_0 - 1)\gamma) \wt{f}^{\, k_0}(\mu) +
                      (k_0 - 1)\gamma \wt{f}^{\, k_0}(\nu).
$$
Now, we define
$$
\mu_{k_0}:= \wt{f}^{\, k_0}(\nu).
$$
For each Borel subset $X$ of $M$,
$$
(\wt{f}(\mu_{k_0-1}))(X)
  \leq 1 - (k_0 - 1)\gamma + (k_0 - 1)\gamma(\wt{f}^{\, k_0}(\nu))(X)
  \leq \mu_{k_0}(X^\gamma) + \gamma.
$$
Thus, $d_P(\mu_{k_0},\wt{f}(\mu_{k_0-1})) \leq \gamma < \delta$.
Finally, it is enough to complete the sequence by defining
$\mu_{k_0+1}:= \wt{f}(\mu_{k_0}),\ldots,\mu_k:= \wt{f}(\mu_{k-1})$.
\end{proof}

\begin{theorem}\label{ChainH}
Given $0 < \delta < 1$, there exists $k_0 \in \N$ such that
for any $h \in \cH(M)$, any $\mu,\nu \in \cM(M)$ and any $k \geq k_0$,
there exist
$$
\mu_1 \in B(\wt{h}(\mu);\delta), \
\mu_2 \in B(\wt{h}(\mu_1);\delta),\ldots,
\mu_k \in B(\wt{h}(\mu_{k-1});\delta)
$$
such that
$$
\mu_k = \nu.
$$
\end{theorem}

\begin{proof}
Let $k_0 \in \N$ be as in Theorem \ref{ChainC}. Since $h \in \cH(M)$,
we can choose $\nu' \in \cM(M)$ such that $\wt{h}^{\, k}(\nu') = \nu$.
Hence, it is enough to consider $\nu'$ in place of $\nu$ in
Theorem~\ref{ChainC}.
\end{proof}

The next result characterizes the chain continuity of the induced map
$\wt{f}$.

\begin{theorem}\label{CCCont}
For every $f \in \cC(M)$, the following assertions are equivalent:
\begin{description}
\item {\rm   (i)} $\wt{f}$ is chain continuous at some point;
\item {\rm  (ii)} $\wt{f}$ is chain continuous at every point;
\item {\rm (iii)} $\bigcap_{n=1}^\infty f^n(M)$ is a singleton.
\end{description}
\end{theorem}

\begin{proof}
(i) $\Rightarrow$ (iii):
Put $Y:= \bigcap_{n=1}^\infty f^n(M)$. By hypothesis, there exists
$\mu \in \cM(M)$ such that $\wt{f}$ is chain continuous at $\mu$.
Fix $0 < \eps < 1/2$. There exists $0 < \delta < 1$ such that the relations
$$
\mu_0 \in B(\mu;\delta), \
\mu_1 \in B(\wt{f}(\mu_0);\delta), \
\mu_2 \in B(\wt{f}(\mu_1);\delta),\ldots
$$
imply
$$
d_P(\mu_n,\wt{f}^{\, n}(\mu)) < \eps \ \ \text{ for all } n \geq 0.
$$
Let $k_0 \in \N$ be associated to this $\delta$ as in Theorem~\ref{ChainC}.
Given $y \in Y$, we can choose $z \in M$ such that $f^{k_0}(z) = y$.
By Theorem~\ref{ChainC} with $\nu = \pi_z$ and $k = k_0$, there exist
$$
\mu_1 \in B(\wt{f}(\mu);\delta), \
\mu_2 \in B(\wt{f}(\mu_1);\delta),\ldots,
\mu_{k_0} \in B(\wt{f}(\mu_{k_0-1});\delta)
$$
with
$$
\mu_{k_0} = \pi_{f^{k_0}(z)} = \pi_y.
$$
Hence, $d_P(\pi_y,\wt{f}^{\, k_0}(\mu)) < \eps$. Since $y \in Y$ is arbitrary,
we conclude that
$$
d_P(\pi_y,\pi_w) < 2\eps \ \ \text{ whenever } y,w \in Y.
$$
This implies that $\diam(Y) < 2\eps$, which proves (iii).

\smallskip
\noindent (iii) $\Rightarrow$ (ii):
Let $\bigcap_{n=1}^\infty f^n(M) = \{a\}$. We claim that
$$
\bigcap_{n=1}^\infty \wt{f}^{\,n}(\cM(M)) = \{\pi_a\}.
$$
Indeed, since $\wt{f}(\pi_a) = \pi_{f(a)} = \pi_a$, it is clear that
$\pi_a$ belongs to the above intersection. Conversely, let $\nu$ be
an element of the above intersection. Then, for each $n \in \N$,
there exists $\mu_n \in \cM(M)$ such that $\nu = \wt{f}^{\,n}(\mu_n)$.
Hence,
$$
\nu(f^n(M)) = \mu_n(f^{-n}(f^n(M))) = \mu_n(M) = 1 \ \ \text{ for all }
n \in \N,
$$
which implies that $\nu(\{a\}) = 1$, that is, $\nu = \pi_a$.
Now, (ii) follows from Lemma~\ref{CC} below.

\smallskip
\noindent (ii) $\Rightarrow$ (i): Obvious.
\end{proof}

\begin{lemma}\label{CC}
If $f \in \cC(M)$ and $\bigcap_{n=1}^\infty f^n(M)$ is a singleton,
then $f$ is chain continuous at every point.
\end{lemma}

\begin{proof}
Assume $\bigcap_{n=1}^\infty f^n(M) = \{a\}$.
Then $a$ is a fixed point of $f$ that uniformly attracts all orbits.
Hence, it is clear that $f$ is equicontinuous at every point.
Fix $\eps > 0$ and let $\gamma > 0$ be such that the relation
$d(y,a) < \gamma$ implies $d(f^n(y),a) < \eps/3$ for all $n \geq 0$.
Let $k \in \N$ be such that $f^k(M) \subset B(a;\gamma)$. Then,
\begin{equation}
d(f^n(y),a) < \frac{\eps}{3} \ \text{ for all } y \in M \text{ and all }
  n \geq k. \label{E1}
\end{equation}
Moreover, there exists $\delta > 0$ such that the following property holds:
\begin{description}
\item {($\ast$)} For every $y \in M$ and for every choice of points
      $$
      y_0 \in B(y;\delta), y_1 \in B(f(y_0);\delta),\ldots,
      y_k \in B(f(y_{k-1});\delta),
      $$
      we have that $y_k \in B(a;\gamma)$ and $d(y_j,f^j(y)) < \eps/3$
      for all $j \in \{0,1,\ldots,k\}$.
\end{description}

Take $x \in M$ and let $x_0 \in B(x;\delta)$, $x_1 \in B(f(x_0);\delta)$,
$x_2 \in B(f(x_1);\delta),\ldots$. We have to prove that
\begin{equation}
d(x_n,f^n(x)) < \eps \ \text{ for all } n \geq 0. \label{E2}
\end{equation}
By ($\ast$), $x_k \in B(a;\gamma)$ and the inequality in (\ref{E2}) holds
for every $n \in \{0,1,\ldots,k\}$. Since $d(x_k,a) < \gamma$,
\begin{equation}
d(f^n(x_k),a) < \frac{\eps}{3} \ \text{ for all } n \geq 0. \label{E3}
\end{equation}
Moreover, by applying ($\ast$) with $y = y_0 = x_k, y_1 = x_{k+1},\ldots,
y_k = x_{2k}$, we see that $x_{2k} \in B(a;\gamma)$ and
\begin{equation}
d(x_n,f^{n-k}(x_k)) < \frac{\eps}{3} \ \text{ for all }
  n \in \{k+1,\ldots,2k\}. \label{E4}
\end{equation}
By (\ref{E1}), (\ref{E3}) and (\ref{E4}), the inequality in (\ref{E2}) holds
for every $n \in \{k+1,\ldots,2k\}$. Since $d(x_{2k},a) < \gamma$,
we can repeat the argument and conclude that $x_{3k} \in B(a;\gamma)$ and
the inequality in (\ref{E2}) holds for every $n \in \{2k+1,\ldots,3k\}$.
By continuing this process, we obtain the desired result.
\end{proof}

For the generic $f \in \contcantor$, since $f$ has no periodic point,
it follows from Theorem~\ref{CCCont} that $\wt{f}$ has no point of chain
continuity.

\begin{theorem}\label{CCHom}
Suppose that $M$ has at least two points. For every $h \in \cH(M)$,
$\wt{h}$ has no point of chain continuity.
\end{theorem}

\begin{proof}
Follows immediately from Theorem~\ref{CCCont}.
\end{proof}

Given $f \in \cC(M)$, recall that $f$ is {\em topologically transitive}
(resp.\ {\em mixing}) if, for any pair $U,V \subset M$ of nonempty open
sets, there exists $k \in \N_0$ (resp.\ $k_0 \in \N_0$) such that
$f^k(U) \cap V \neq \emptyset$ (resp.\ for all $k \geq k_0$).

\smallskip
Given $f \in \cC(M)$ and $\delta > 0$, recall that a finite sequence
$(x_n)_{n=0,1,\ldots,k}$ of elements of $M$ is a {\em $\delta$-chain from
$x_0$ to $x_k$} if $d(f(x_n),x_{n+1}) < \delta$ for all $n=0,1,\ldots,k-1$.
In this case, we say that $k$ is the {\em length} of the chain.
Recall that $f$ is {\em chain mixing} if for every $\delta > 0$ and
for every pair $x,y \in M$, there exists $k_0 \in \N$ such that
for all $k \geq k_0$, there exists a $\delta$-chain from $x$ to $y$
of length $k$. Note that if $f$ is chain mixing, then $f$ is necessarily
surjective.

\begin{theorem}\label{ChainMixing}
For every $h \in \cH(M)$, $\wt{h}$ is chain mixing.
\end{theorem}

\begin{proof}
Follows immediately from Theorem \ref{ChainH}.
\end{proof}

In view of the above theorem, it is natural to ask if $\wt{h}$ is always
mixing. Let us see that this is not the case. Indeed,
it was proved in \cite{WBauKSig75} that $\wt{h}$ topologically transitive
implies $h$ topologically transitive. As a consequence, for the
generic $h \in \homcantor$, $\wt{h}$ is not topologically transitive.

\smallskip
On the other hand, for the generic $f \in \contcantor$, since $f$ is not
surjective, it follows that $\wt{f}$ is neither topologically transitive nor
chain mixing.

\smallskip
Given $h \in \cH(M)$, recall that a sequence $(x_n)_{n \in \Z}$ is a
{\em $\delta$-pseudotrajectory} ($\delta > 0$) of $h$ if
$$
d(h(x_n),x_{n+1}) \leq \delta \ \ \text{ for all } n \in \Z.
$$
The homeomorphism $h$ has the {\em weak shadowing property} \cite{RCorSPil95}
if for every $\eps > 0$ there exists $\delta > 0$ such that for every
$\delta$-pseudotrajectory $(x_n)_{n \in \Z}$ of $h$ there exists $x \in M$
such that
$$
\{x_n : n \in \Z\} \subset \{h^n(x) : n \in \Z\}^\eps.
$$
Moreover, $h$ has the {\em shadowing property} \cite{RBow75a,RBow75b}
if for every $\eps > 0$ there exists $\delta > 0$ such that every
$\delta$-pseudotrajectory $(x_n)_{n \in \Z}$ of $h$ is {\em $\eps$-shadowed}
by a real trajectory of $h$, i.e., there exists $x \in M$ such that
$$
d(x_n,h^n(x)) < \eps \ \ \text{ for all } n \in \Z.
$$

It was proved in \cite{NBerRVer14} that for the generic $h \in \homcantor$,
the induced map $\ov{h}$ has the shadowing property. Again, we shall see
that the induced map $\wt{h}$ has a completely different behaviour.

\begin{theorem}\label{Shadowing}
For the generic $h \in \homcantor$, $\wt{h}$ does not have the weak shadowing
property.
\end{theorem}

\begin{proof}
Fix a generic $h \in \homcantor$ and suppose that $\wt{h}$ has the weak
shadowing property. Let $U,V$ be a pair of nonempty open sets in
$\mcantor$. Fix $\mu \in U$ and $\nu \in V$, and choose $\eps > 0$
such that
$$
B(\mu;\eps) \subset U \ \ \ \ \text{ and } \ \ \ \
B(\nu;\eps) \subset V.
$$
Since $\wt{h}$ has the weak shadowing property, there is a $\delta > 0$
associated to this $\eps$ according to the definition of weak shadowing.
Since $\wt{h}$ is chain mixing (Theorem~\ref{ChainMixing}), there is a
$\delta$-chain $(\mu_0,\mu_1,\ldots,\mu_k)$ of $\wt{h}$ starting at
$\mu_0 = \mu$ and ending at $\mu_k = \nu$. Of course, we can extend
this $\delta$-chain to a $\delta$-pseudotrajectory $(\mu_n)_{n \in \Z}$
of $\wt{h}$. By weak shadowing, there exists $\eta \in \mcantor$ such that
$$
\{\mu_n : n \in \Z\} \subset \{\wt{h}^{\, n}(\eta) : n \in \Z\}^\eps.
$$
In particular, there are $n_1,n_2 \in \Z$ such that
$\wt{h}^{\, n_1}(\eta) \in U$ and $\wt{h}^{\, n_2}(\eta) \in V$,
and so $\wt{h}^{\, n_2-n_1}(U) \cap V \neq \emptyset$.
This implies that $\wt{h}$ is topologically transitive.
As observed after the proof of Theorem~\ref{ChainMixing},
this is a contradiction.
\end{proof}

\begin{remark}
The above proof actually establishes the following more general result:
\begin{quote}
{\it If $h \in \cH(M)$ is not topologically transitive,
then $\wt{h}$ does not have the weak shadowing property.}
\end{quote}
\end{remark}

%%%%%%%%%%%%%%%%%%%%%%%%%%%%%%%%%%%%%%%%%%%%%%%%%%%%%%%%%%%%%%%%%%%%%%%%%%%%%

\section{Recurrence}

\hspace*{4mm} Given a map $f : M \to M$, we denote by $P(f)$ (resp.\
$R(f)$, $\Omega(f)$, $CR(f)$), the set of all periodic (resp.\ recurrent,
nonwandering, chain recurrent) points of $f$.

\smallskip
Given partitions $\p$ and $\q$ of $\cantor$, we say that $\q$
{\em strongly refines} $\p$ if each $a \in \q$ is properly contained
in some $a' \in \p$. A sequence $(\p_n)$ of partitions of $\cantor$
is said to be {\em strongly decreasing} if $\p_{n+1}$ strongly refines $\p_n$
for all $n$. Recall that $(\p_n)$ is said to be {\em null} if
$\mesh(\p_n) \to 0$ as $n \to \infty$ \cite{NBerUDar12}.
Note that every null sequence of partitions of $\cantor$
has a strongly decreasing (and null) subsequence.

\begin{theorem}\label{PeriodicC}
For the generic $f \in \contcantor$, the following properties hold:
\begin{description}
\item {\rm (a)} $\wt{f}$ has uncountably many periodic points
      of each period $p \geq 1$.
\item {\rm (b)} Any neighborhood of any periodic point of $\wt{f}$
      of period $p$ contains uncountably many periodic points of $\wt{f}$
      of period $kp$, for each $k \in \N$.
\item {\rm (c)} $R(\wt{f}) = \Omega(\wt{f}) = CR(\wt{f})$.
\item {\rm (d)} $CR(\wt{f})$ has empty interior in $\wt{f}(\cM(\cantor))$.
\item {\rm (e)} $P(\wt{f})$ is dense in $CR(\wt{f})$.
\end{description}
\end{theorem}

\begin{proof}
Let $f \in \contcantor$ satisfy property (Q). We shall divide the proof
in seven steps.

\medskip
\noindent {\bf Step 1.} Let $\p$ be a partition of $\cantor$ such that
every component of $\gr(f,\p)$ is a balloon of a certain type $(q!,q!)$.
Choose one such component $B$; say
$$
B = \{v_1,\ldots,v_{q!}\} \cup \{w_1,\ldots,w_{q!}\},
$$
with usual labeling. Choose also a vertex $w \in \{w_1,\ldots,w_{q!}\}$
and an integer $p \in \N$. Let $k \in \N$ be the smallest integer such that
$$
f^{kp}(w) \subset w.
$$
Let $a_0,\ldots,a_{k-1} \in \{w_1,\ldots,w_{q!}\}$ be determined by the
relations
$$
f^{jp}(w) \subset a_j, \ \ \ j = 0,\ldots,k-1.
$$
Then, there are uncountably many periodic points $\mu$ of $\wt{f}$
of period $p$ such that
\begin{equation}
\mu(a_0) = \mu(a_1) = \cdots = \mu(a_{k-1}) \label{A}
\end{equation}
and
\begin{equation}
\mu(a) = 0 \ \ \text{ for all }
           a \in \p \backslash \{a_0,a_1,\ldots,a_{k-1}\}. \label{B}
\end{equation}

\medskip
Indeed, let us fix a strongly decreasing null sequence $(\p_n)_{n \in \N}$
of partitions of $\cantor$ such that each component of $\gr(f,\p_n)$
is a balloon of type $(q_n!,q_n!)$, $\p_1$ refines $\p$ and $q_1 > p$.
By an {\em admissible sequence} we mean a sequence $\cB:= (B_n)_{n \in \N}$,
where each $B_n$ is a component (balloon) of $\gr(f,\p_n)$, such that
the initial vertex of $B_1$ is contained in the initial vertex of $B$ and
the initial vertex of $B_{n+1}$ is contained in the initial vertex of $B_n$
for each $n \in \N$. To each admissible sequence $\cB$, we shall associate
a measure $\mu_\cB \in \mcantor$ which will be constructed as follows. Write
$$
B_n = \{v_{n,1},\ldots,v_{n,q_n!}\} \cup \{w_{n,1},\ldots,w_{n,q_n!}\},
$$
with usual labeling. We extend the ``loop'' $\{w_{n,1},\ldots,w_{n,q_n!}\}$
to a sequence $(w_{n,j})_{j \in \N}$ by considering $w_{n,i} = w_{n,j}$
whenever $i \equiv j$ mod $q_n!$. It is easy to verify that the collection
$$
\cS:= \{\emptyset\} \cup \p_1 \cup \p_2 \cup \p_3 \cup \ldots
$$
is a semiring of subsets of $\cantor$ (i.e., $a,b \in \cS$ implies that
$a \cap b \in \cS$ and that $a \backslash b$ is a finite union of pairwise
disjoint elements of $\cS$).
Since $\p_1$ refines $\p$ and $v_{1,1} \subset v_1$, it follows that
$$
q_1 \geq q \ \ \ \text{ and } \ \ \ w_{1,1} \subset w_1.
$$
Moreover, since $\p_{n+1}$ refines $\p_n$ and $v_{n+1,1} \subset v_{n,1}$,
we also have that
$$
q_{n+1} \geq q_n \ \ \ \text{ and } \ \ \ w_{n+1,1} \subset w_{n,1}
  \ \ \ \ \ (n \in \N).
$$
As a consequence, there is a smallest $t \in \N$ such that
$$
w_{n,t} \subset w \ \ \text{ for all } n \in \N.
$$
We define a set function $\varphi : \cS \to [0,1]$ by
$$
\varphi(a):= \frac{p}{q_n!} \ \ \text{ if } a = w_{n,t+jp} \text{ for some }
  n \in \N \text{ and some } 0 \leq j \leq \frac{q_n!}{p} - 1
$$
and
$$
\varphi(a):= 0 \ \ \text{ otherwise}.
$$

We claim that
\begin{equation}
\varphi(a) = \sum_{b \in I_{\p_{n+1}}(a)} \varphi(b), \label{C}
\end{equation}
for every $n \in \N$ and every $a \in \p_n$.
Indeed, for each $0 \leq j \leq \frac{q_n!}{p} - 1$,
$$
w_{n+1,t+ip} \subset w_{n,t+jp} \ \Longleftrightarrow \
i = j + \ell \,\frac{q_n!}{p}
\text{ for some } 0 \leq \ell \leq \frac{q_{n+1}!}{q_n!} - 1.
$$
Therefore,
$$
\sum_{b \in I_{\p_{n+1}}(w_{n,t+jp})} \varphi(b)
  = \sum_{\ell = 0}^{\frac{q_{n+1}!}{q_n!} - 1} \varphi(w_{n+1,t+jp+\ell q_n!})
  = \frac{q_{n+1}!}{q_n!} \cdot \frac{p}{q_{n+1}!}
  = \frac{p}{q_n!}
  = \varphi(w_{n,t+jp}).
$$
On the other hand, if $a \neq w_{n,t+jp}$ for all $0 \leq j \leq
\frac{q_n!}{p} - 1$, then no set of the form $w_{n+1,t+ip}$
($0 \leq i \leq \frac{q_{n+1}!}{p} - 1$) is contained in $a$, and so
$$
\varphi(a) = 0 = \sum_{b \in I_{\p_{n+1}}(a)} \varphi(b).
$$
This completes the proof of our claim.

\smallskip
Let us now prove that $\varphi$ is finitely additive.
Let $a \in \cS$ be nonempty and assume that $a$ is the union of a finite
collection $\cC$ of pairwise disjoint nonempty elements of $\cS$.
We have to prove that
\begin{equation}
\varphi(a) = \sum_{c \in \cC} \varphi(c). \label{D}
\end{equation}
Since this is obvious if $\cC = \{a\}$, let us assume that this is not
the case. Let $n \in \N$ be such that $a \in \p_n$ and let $m \in \N$ be
the largest positive integer such that $\cC \cap \p_{n+m} \neq \emptyset$.
Define
$$
\cC_j:= \cC \cap \p_{n+j} \ \ \text{ for } j = 1,\ldots,m.
$$
Then $\cC = \cC_1 \cup \ldots \cup \cC_m$. We shall prove (\ref{D})
by induction on $m$. If $m = 1$ then $\cC = \cC_1 = I_{\p_{n+1}}(a)$,
and so (\ref{D}) follows from (\ref{C}). Assume $m \geq 2$ and the result
true with $m-1$ in place of $m$. Choose $b \in \cC_m$. There is a unique
$b' \in \p_{n+m-1}$ such that $b \subset b'$. Since $\p_{n+m}$ strongly
refines $\p_{n+m-1}$, $b' \neq b$. Moreover, since $b \subset a$,
we must have $b' \subset a$. Thus, $b' \not\in \cC$ and
$I_{\p_{n+m}}(b') \subset \cC$. We define
$$
\cC':= (\cC \backslash I_{\p_{n+m}}(b')) \cup \{b'\}.
$$
Then, $a = \cup \cC'$ (with disjoint union) and
$\sum_{c' \in \cC'} \varphi(c') = \sum_{c \in \cC} \varphi(c)$
because of (\ref{C}). We can repeat this argument until we obtain
a finite collection $\cD$ of pairwise disjoint nonempty elements of $\cS$
such that
$$
a = \cup \cD, \ \ \ \
\sum_{d \in \cD} \varphi(d) = \sum_{c \in \cC} \varphi(c) \ \ \
\text{ and } \ \ \
\cD \cap \p_{n+j} = \emptyset \ \text{ for all } j \geq m.
$$
By the induction hypothesis,
$$
\sum_{c \in \cC} \varphi(c) = \sum_{d \in \cD} \varphi(d) = \varphi(a),
$$
as was to be shown.

\smallskip
Since the elements of $\cS$ are clopen, it is not possible to write an
element of $\cS$ as a countably infinite union of pairwise disjoint nonempty
elements of $\cS$. As a consequence, $\varphi$ is countably additive.
By the extension theorem of measure theory,
there exists a measure $\mu_\cB$ defined on a $\sigma$-algebra $\cA$
containing $\cS$ that extends $\varphi$. Since every open subset of $\cantor$
can be written as a countable union of elements of $\cS$, $\cA$ contains
the Borel subsets of $\cantor$. Hence, we may regard $\mu_\cB$ as a Borel
measure. Since
$$
\mu_\cB(\cantor) = \sum_{a \in \p_1} \mu_\cB(a)
  = \sum_{a \in \p_1} \varphi(a)
  = \sum_{j=0}^{\frac{q_1!}{p}-1} \varphi(w_{1,t+jp})
  = \sum_{j=0}^{\frac{q_1!}{p}-1} \frac{p}{q_1!}
  = 1,
$$
we see that $\mu_\cB$ is a probability measure. In other words,
$$
\mu_\cB \in \mcantor.
$$
For each $n \in \N$, since $\mu_\cB(f^{-p}(a)) = \mu_\cB(a)$ for all
$a \in \p_n$, it follows from Lemma~\ref{Lemma2} that
$$
d_P(\wt{f}^{\, p}(\mu_\cB),\mu_\cB) \leq \mesh \p_n.
$$
Since $\mesh \p_n \to 0$ as $n \to \infty$, we obtain
$$
\wt{f}^{\, p}(\mu_\cB) = \mu_\cB.
$$
On the other hand,
$\mu_\cB(f^{-j}(w_{1,t+p})) = 0 \neq p/q_1! = \mu_\cB(w_{1,t+p})$
for each $1 \leq j < p$, which implies that
$$
\wt{f}^{\, j}(\mu_\cB) \neq \mu_\cB \ \ \text{ for each } 1 \leq j < p.
$$
Therefore, $\mu_\cB$ is a periodic point of $\wt{f}$ of period $p$.

\smallskip
If $\cB':= (B'_n)_{n \in \N}$ is an admissible sequence with $\cB' \neq \cB$,
then $B'_m \neq B_m$ for some $m \in \N$, and so $\mu_\cB(\cup B_m) = 1$
whereas $\mu_{\cB'}(\cup B_m) = 0$. This shows that distinct admissible
sequences generate distinct probability measures. Since the set of all
admissible sequences is uncountable (by a simple diagonal argument),
we conclude that
$$
\{\mu_\cB : \cB \text{ is an admissible sequence}\}
$$
is an uncountable set of periodic points of $\wt{f}$ of period $p$.
By construction, it is easy to see that each $\mu_\cB$ satisfies
(\ref{A}) and (\ref{B}). Thus, the proof of Step 1 is complete.

\medskip
\noindent {\bf Step 2.} Proof of (a).

\medskip
Property (a) follows immediately from Step 1.

\medskip
\noindent {\bf Step 3.} Let $\p$ be a partition of $\cantor$ such that
every component of $\gr(f,\p)$ is a balloon of type $(q!,q!)$. Let
$$
B_i := \{v_{i,1},\ldots,v_{i,q!}\} \cup \{w_{i,1},\ldots,w_{i,q!}\}
  \ \ \ (1 \leq i \leq N)
$$
be the components (balloons) of $\gr(f,\p)$. For every $\mu \in CR(\wt{f})$,
$$
\mu(v_{i,j}) = 0 \ \ \text{ for all } i \text{ and } j.
$$

\medskip
Choose $0 < \delta < \delta(\p)$.
Since $\mu$ is a chain recurrent point of $\wt{f}$, there is a $\delta$-chain
$(\mu_n)_{n=0,1,\ldots,k}$ from $\mu_0:= \mu$ to $\mu_k:= \mu$. Since
$d_P(\wt{f}(\mu_n),\mu_{n+1}) < \delta$, Lemma~\ref{Lemma1} gives
$$
|\mu_n(f^{-1}(a)) - \mu_{n+1}(a)| < \delta \ \ \text{ for all } a \in \p
 \ \ (0 \leq n \leq k-1).
$$
Since $f^{-q!}(v_{i,j}) = \emptyset$, it follows that
$$
\mu(v_{i,j}) < q!\delta \ \ \ \ \ (1 \leq i \leq N, 1 \leq j \leq q!).
$$
Since $\delta >0$ can be chosen arbitrarily small, 
$\mu(v_{i,j}) = 0$ for each $i$ and $j$.

\medskip
\noindent {\bf Step 4.} Proof of (c).

\medskip
Fix $\mu \in CR(\wt{f})$. Given $\eps > 0$, let $\p$ be a partition of
$\cantor$ of mesh $< \eps$ such that every component of $\gr(f,\p)$ is a
balloon of type $(q!,q!)$. Let
$$
B_i := \{v_{i,1},\ldots,v_{i,q!}\} \cup \{w_{i,1},\ldots,w_{i,q!}\}
  \ \ \ (1 \leq i \leq N)
$$
be the components (balloons) of $\gr(f,\p)$. Since
$f^{-q!}(v_{i,j}) = \emptyset$ and $f^{-q!}(w_{i,j}) = v_{i,j} \cup w_{i,j}$,
it follows from Step~3 that
$$
\mu(f^{-q!}(a)) = \mu(a) \ \ \text{ for all } a \in \p.
$$
Thus, by Lemma~\ref{Lemma2}, $d_P(\wt{f}^{\, q!}(\mu),\mu)) < \eps$.
This proves that $\mu \in R(\wt{f})$.

\medskip
\noindent {\bf Step 5.} Proof of (d).

\medskip
Let $\mu:= \wt{f}(\nu)$ be an arbitrary element in the range of $\wt{f}$.
Let $\p$ be a partition of $\cantor$ such that every component of $\gr(f,\p)$
is a balloon of type $(q!,q!)$ with $q \geq 2$. Choose one such component
$$
B:= \{v_1,\ldots,v_{q!}\} \cup \{w_1,\ldots,w_{q!}\}
$$
and choose a point $z \in v_1$. For each $\lambda \in (0,1)$, define
$$
\mu_\lambda:= (1-\lambda)\mu + \lambda \pi_{f(z)}.
$$
Note that each $\mu_\lambda$ belons to the range of $\wt{f}$ because
$\mu_\lambda = \wt{f}((1-\lambda)\nu + \lambda \pi_z))$. Since
$\mu_\lambda(v_2) \geq \lambda > 0$, it follows from Step~3 that
$\mu_\lambda \not\in CR(\wt{f})$. Moreover, by Lemma~\ref{Lemma4},
$d_P(\mu_\lambda,\mu) \leq \lambda$.
Thus, there are points of $\wt{f}(\cM(\cantor)) \backslash CR(\wt{f})$
arbitrarily close to $\mu$.

\medskip
\noindent {\bf Step 6.} Proof of (e).

\medskip
Fix $\mu \in R(\wt{f})$ and $\eps > 0$. Let $\p$ and $B_1,\ldots,B_N$
be as in Step~4. Define
$$
\delta := \min\Big\{\delta(\p),\frac{\mesh(\p)}{2q!\card(\p)}\Big\}.
$$
Since $\mu$ is a recurrent point of $\wt{f}$, there exists $p \in \N$
such that $d_P(\wt{f}^{\, p}(\mu),\mu) < \delta$. By Lemma~\ref{Lemma1},
\begin{equation}
|\mu(f^{-p}(a)) - \mu(a)| < \delta \ \ \text{ for all } a \in \p. \label{Equa1}
\end{equation}
Moreover, by Step~3,
\begin{equation}
\mu(v_{i,j}) = 0 \ \ \text{ for all } i \text{ and } j. \label{Equa2}
\end{equation}
Define
$$
W_i:= w_{i,1} \cup \ldots \cup w_{i,q!} \ \ \ (1 \leq i \leq N).
$$
If $a$ is some $w_{i,j}$, then the sequence
$$
a, \ f^{-p}(a) \cap W_i, \ f^{-2p}(a) \cap W_i,\ldots
$$
is periodic. Let $k$ be the period of this sequence. Note that $k$ does not
depend on $i$ or $j$. Hence, each ``loop'' $\{w_{i,1},\ldots,w_{i,q!}\}$
can be partitioned in sets
$$
\{a_{i,r,0},a_{i,r,1},\ldots,a_{i,r,k-1}\} \ \ \ (1 \leq r \leq q!/k)
$$
satisfying
$$
a_{i,r,t} = f^{-tp}(a_{i,r,0}) \cap W_i \
  \text{ for } 1 \leq t \leq k - 1
$$
and
$$
a_{i,r,0} = f^{-kp}(a_{i,r,0}) \cap W_i.
$$
By (\ref{Equa1}) and (\ref{Equa2}),
\begin{equation}
|\mu(a_{i,r,t}) - \mu(a_{i,r,0})| < t\delta < q!\delta \
  \text{ for all } 1 \leq t \leq k-1.
  \label{Equa3}
\end{equation}
Define
$$
d_{i,r}:= \mu(a_{i,r,0}) + \mu(a_{i,r,1}) + \cdots + \mu(a_{i,r,k-1}).
$$
It follows from (\ref{Equa3}) that
\begin{equation}
\Big| \frac{d_{i,r}}{k} - \mu(a_{i,r,0}) \Big| < q!\delta.
  \label{Equa4}
\end{equation}
By Step 1, for each $ 1 \leq i \leq N$ and each $1 \leq r \leq q!/k$,
there are uncountably many periodic points $\mu_{i,r}$ of $\wt{f}$
of period $p$ such that
$$
\mu_{i,r}(a_{i,r,0}) = \mu_{i,r}(a_{i,r,1}) = \cdots = \mu_{i,r}(a_{i,r,k-1})
$$
and
$$
\mu_{i,r}(a) = 0 \ \text{ for all }
  a \in \p \backslash \{a_{i,r,0},a_{i,r,1},\ldots,a_{i,r,k-1}\}.
$$
Without loss of generality, we may assume $d_{1,1} \neq 0$. We fix one such
periodic point $\mu_{i,r}$ for each $(i,r) \neq (1,1)$ and consider the
uncountably many possibilities for the periodic point $\mu_{1,1}$.
In this way we have uncountably many periodic points of $\wt{f}$ of the
form
\begin{equation}
\mu':= \sum_{i=1}^N \sum_{r=1}^{q!/k} d_{i,r} \mu_{i,r}, \label{Equa5}
\end{equation}
satisfying
$$
\wt{f}^{\, p}(\mu') = \mu'.
$$
Given $1 \leq j < p$, it is not possible that two of these periodic points
have period $j$. Indeed, assume that
$$
d_{1,1} \mu_{1,1} + \sum_{(i,r) \neq (1,1)} d_{i,r} \mu_{i,r}
\ \ \ \ \text{ and } \ \ \ \
d_{1,1} \mu'_{1,1} + \sum_{(i,r) \neq (1,1)} d_{i,r} \mu_{i,r}
$$
have period $j$. Then
\begin{equation}
\wt{f}^{\, j}(\mu_{1,1}) - \wt{f}^{\, j}(\mu'_{1,1})
  = \mu_{1,1} - \mu'_{1,1}. \label{Equa6}
\end{equation}
By the way the measures are constructed in Step 1,
$\mu_{1,1}$ and $\mu'_{1,1}$ correspond to distinct admissible sequences
and so there is a Borel set $b$ in $\cantor$ such that
$$
\mu_{1,1}(b) > 0, \ \mu_{1,1}(f^{-j}(b)) = 0 \ \ \text{ and } \ \
\mu'_{1,1}(b) = \mu'_{1,1}(f^{-j}(b)) = 0,
$$
which contradicts (\ref{Equa6}). Thus, uncountably many measures of the
form (\ref{Equa5}) are periodic points of $\wt{f}$ of period $p$.
Let $\mu'$ be as in (\ref{Equa5}). We shall prove that
\begin{equation}
|\mu'(a) - \mu(a)| < \frac{\mesh(\p)}{\card(\p)} \ \
  \text{ for all } a \in \p. \label{Equa7}
\end{equation}
Indeed, if $a$ is some $v_{i,j}$, then $\mu'(a) = 0 = \mu(a)$ because of
(\ref{Equa2}). If $a$ is some $w_{i,j}$, then there are unique
$1 \leq r \leq q!/k$ and $0 \leq t \leq k-1$ such that $a = a_{i,r,t}$.
Hence,
$$
\mu'(a) = d_{i,r} \mu_{i,r}(a) = \frac{d_{i,r}}{k}\cdot
$$
By (\ref{Equa3}) and (\ref{Equa4}), $|\mu'(a) - \mu(a)| < 2q!\delta$.
By our choice of $\delta$, we also obtain the inequality in (\ref{Equa7})
in this case. Therefore, Lemma~\ref{Lemma2} tell us that
$d_P(\mu',\mu) < \eps$, as was to be shown.

\medskip
\noindent {\bf Step 7.} Proof of (b).

\medskip
In the proof of Step 6, if $\mu$ is already a periodic point of $\wt{f}$,
of period $t$ say, then the $p$ can be chosen to be any multiple $kt$ of $t$,
and so we obtain uncountably many periodic points of $\wt{f}$ of period $kt$
in the $\eps$-neighborhood of $\mu$.
\end{proof}

Let us now see that a result similar to Theorem~\ref{PeriodicC} holds
in the case of homeomorphisms.

\begin{theorem}\label{PeriodicH}
For the generic $h \in \homcantor$, the following properties hold:
\begin{description}
\item {\rm (a)} $\wt{h}$ has uncountably many periodic points
      of each period $p \geq 1$.
\item {\rm (b)} Any neighborhood of any periodic point of $\wt{h}$
      of period $p$ contains uncountably many periodic points of $\wt{h}$
      of period $kp$, for each $k \in \N$.
\item {\rm (c)} $R(\wt{h}) = \Omega(\wt{h})$ and $CR(\wt{h}) = \cM(\cantor)$.
\item {\rm (d)} $\Omega(\wt{f})$ has empty interior in $\cM(\cantor)$.
\item {\rm (e)} $P(\wt{h})$ is dense in $\Omega(\wt{h})$.
\end{description}
\end{theorem}

\begin{proof}
Let $h \in \homcantor$ satisfy property (P). We choose a partition $\p$
of $\cantor$ such that every component of $\gr(h,\p)$ is a balanced dumbbell
with plate weight $q!$. Let
$$
D_i := \{u_{i,1},\ldots,u_{i,q!}\} \cup \{v_{i,1},\ldots,v_{i,s_i}\}
  \cup \{w_{i,1},\ldots,w_{i,q!}\} \ \ \ (1 \leq i \leq N)
$$
be the components (dumbbells) of $\gr(h,\p)$.
We claim that, for every $\mu \in \Omega(\wt{h})$,
\begin{equation}
\mu(v_{i,j}) = 0 \ \ \text{ and } \ \ \mu(h^{-n}(v_{i,1})) = 0 \
  \text{ for all } i, j \text{ and } n. \label{1}
\end{equation}
Indeed, let us fix $1 \leq i \leq N$ and $m \in \N$. It is enough to prove
that $\mu(h^{-m}(v_{i,s_i})) = 0$. For this purpose, let $b$ be the unique
element of $\p$ that contains $h^{-m}(v_{i,s_i})$. If $b$ is a certain
$v_{i,j}$, then $h^{-m}(v_{i,s_i}) = v_{i,j}$ and we define $\p':= \p$.
Otherwise, we define $\p'$ as the partition of $\cantor$ obtained from
$\p$ by replacing $b$ by the sets $h^{-m}(v_{i,s_i})$ and
$b \backslash h^{-m}(v_{i,s_i})$. Let $0 < \delta < \delta(\p')$.
Since $\mu$ is a nonwandering point of $\wt{h}$, there exist $t \in \N$
and $\nu \in \cM(\cantor)$ such that
$$
d_P(\nu,\mu) < \frac{\delta}{2} \ \ \ \ \text{ and } \ \ \ \
d_P(\wt{h}^{\, t}(\nu),\mu) < \frac{\delta}{2}\cdot
$$
In particular, $d_P(\wt{h}^{\, t}(\nu),\nu) < \delta$. By Lemma~\ref{Lemma1},
$$
|\nu(h^{-nt}(a)) - \nu(a)| < n\delta \ \
  \text{ for all } a \in \p \text{ and } n \in \N.
$$
Let $k \in \N$ be the smallest positive integer such that
$h^{-kt}(w_{i,j}) \supset w_{i,j}$ for all $j$. There are $1 \leq j \leq q!$
and $1 \leq n \leq m$ such that
$$
h^{-nkt}(w_{i,j}) \supset h^{-m}(v_{i,s_i}).
$$
Since $h^{-nkt}(w_{i,j}) \supset w_{i,j}$ and
$|\nu(h^{-nkt}(w_{i,j})) - \nu(w_{i,j})| < nk\delta$,
we conclude that
$$
\nu(h^{-m}(v_{i,s_i})) < nk\delta \leq mq!\delta.
$$
Since $d_P(\nu,\mu) < \delta < \delta(\p')$, Lemma~\ref{Lemma1} implies
that
$$
|\nu(h^{-m}(v_{i,s_i})) - \mu(h^{-m}(v_{i,s_i}))| < \delta.
$$
Thus, $\mu(h^{-m}(v_{i,s_i})) < (mq! + 1)\delta$. Since $\delta > 0$ can be
chosen arbitrarily small,
$$
\mu(h^{-m}(v_{i,s_i})) = 0,
$$
as was to be shown.

\smallskip
Let $\mu \in \Omega(\wt{h})$. Given $\eps > 0$, we may assume that $\p$ was
chosen with $\mesh \p < \eps$. It follows from (\ref{1}) that
$$
\mu(h^{-q!}(a)) = \mu(a) \ \ \text{ for all } a \in \p.
$$
Hence, $d_P(\wt{h}^{\, q!}(\mu),\mu) < \eps$ by Lemma~\ref{Lemma2},
which shows that $\mu \in R(\wt{h})$. On the other hand, it follows
immediately from Theorem~\ref{ChainMixing} that $CR(\wt{h}) = \cM(M)$
for all $h \in \cH(M)$. This proves property (c).

\smallskip
Let $\mu \in \cM(\cantor)$ be arbitrary and choose a point $z \in v_{1,1}$.
For each $\lambda \in (0,1)$, define
$$
\mu_\lambda:= (1 - \lambda) \mu + \lambda \pi_z.
$$
Then, $\mu_\lambda \not\in \Omega(\wt{h})$ (because of (\ref{1})) and
$d_P(\mu_\lambda,\mu) \leq \lambda$ (by Lemma\ref{Lemma4}).
This implies pro\-perty~(d).

\smallskip
Now, let $\mu \in R(\wt{h})$ and $\eps > 0$. We may assume that $\p$ was
chosen with $\mesh \p < \eps$. Fix a number $\delta$ satisfying
$0 < \delta < \delta(\p)$. Since $\mu$ is a recurrent point of $\wt{h}$,
there exists $p \in \N$ such that $d_P(\wt{h}^{\, p}(\mu),\mu) < \delta$.
Hence, it follows from Lemma~\ref{Lemma1} that
\begin{equation}
|\mu(h^{-np}(a)) - \mu(a)| < n\delta \ \
  \text{ for all } a \in \p \text{ and } n \in \N. \label{2}
\end{equation}
As in the proof of Step 6 in Theorem~\ref{PeriodicC},
each ``right loop'' $\{w_{i,1},\ldots,w_{i,q!}\}$ can be partitioned in sets
$$
\cA_{i,r}:= \{a_{i,r,0},a_{i,r,1},\ldots,a_{i,r,k-1}\} \ \ \ (1 \leq r \leq q!/k)
$$
satisfying
\begin{equation}
a_{i,r,t} = h^{-tp}(a_{i,r,0}) \cap W_i \
  \text{ for } 1 \leq t \leq k - 1 \label{3}
\end{equation}
and
\begin{equation}
a_{i,r,0} = h^{-kp}(a_{i,r,0}) \cap W_i, \label{4}
\end{equation}
where $W_i:= w_{i,1} \cup \ldots \cup w_{i,q!}$ ($1 \leq i \leq N$).
By (\ref{1}), (\ref{2}) and (\ref{3}), $\mu$ is almost constant on
$\cA_{i,r}$. Define
$$
d_{i,r}:= \mu(a_{i,r,0}) + \mu(a_{i,r,1}) + \cdots + \mu(a_{i,r,k-1}).
$$
Then the average $d_{i,r}/k$ is very close to the values of $\mu$ on the
elements of $\cA_{i,r}$. More precisely, by choosing $\delta$ small enough,
we can make the numbers
$$
\Big| \frac{d_{i,r}}{k} - \mu(a_{i,r,t}) \Big|
$$
as small as we want.
Now, let us look at the ``left loop'' $\{u_{i,1},\ldots,u_{i,q!}\}$.
Each of these ``loops'' can be partitioned in sets
$$
\cB_{i,r}:= \{b_{i,r,0},b_{i,r,1},\ldots,b_{i,r,k-1}\} \ \ \ (1 \leq r \leq q!/k)
$$
satisfying
\begin{equation}
h^{-tp}(b_{i,r,0}) \subset b_{i,r,t} \
  \text{ for } 1 \leq t \leq k - 1 \label{5}
\end{equation}
and
\begin{equation}
h^{-kp}(b_{i,r,0}) \subset b_{i,r,0}. \label{6}
\end{equation}
Note that
$$
h^{-kp}(u_{i,j}) \subset u_{i,j} \ \ \ (1 \leq i \leq N, 1 \leq j \leq q!).
$$
Moreover, by (\ref{1}),
$$
\mu\big(u_{i,j} \backslash h^{-kp}(u_{i,j})\big) = 0.
$$
This fact together with (\ref{2}) and (\ref{5}) imply that $\mu$ is almost
constant on $\cB_{i,r}$. Define
$$
e_{i,r}:= \mu(b_{i,r,0}) + \mu(b_{i,r,1}) + \cdots + \mu(b_{i,r,k-1}).
$$
Then the average $e_{i,r}/k$ is very close to the values of $\mu$ on the
elements of $\cB_{i,r}$. More precisely, by choosing $\delta$ small enough,
we can make the numbers
$$
\Big| \frac{e_{i,r}}{k} - \mu(b_{i,r,t}) \Big|
$$
as small as we want. Now, by making a construction similar to the one
in Step 1 of Theorem~\ref{PeriodicC}, we obtain uncountably many periodic
points $\mu_{i,r}$ of $\wt{h}$ of period $p$ such that
$$
\mu_{i,r}(a_{i,r,0}) = \mu_{i,r}(a_{i,r,1}) = \cdots = \mu_{i,r}(a_{i,r,k-1})
$$
and
$$
\mu_{i,r}(a) = 0 \ \text{ for all }
  a \in \p \backslash \{a_{i,r,0},a_{i,r,1},\ldots,a_{i,r,k-1}\}.
$$
Similarly, we can construct uncountably many periodic points $\nu_{i,r}$
of $\wt{h}$ of period $p$ such that
$$
\nu_{i,r}(b_{i,r,0}) = \nu_{i,r}(b_{i,r,1}) = \cdots = \nu_{i,r}(b_{i,r,k-1})
$$
and
$$
\nu_{i,r}(b) = 0 \ \text{ for all }
  b \in \p \backslash \{b_{i,r,0},b_{i,r,1},\ldots,b_{i,r,k-1}\}.
$$
As in the proof of Step 6 of Theorem~\ref{PeriodicC}, we see that
uncountably many measures of the form
$$
\mu':= \sum_{i=1}^N \sum_{r=1}^{q!/k} (d_{i,r}\mu_{i,r} + e_{i,r}\nu_{i,r})
$$
are periodic points of $\wt{h}$ of period $p$. Moreover, by choosing
$\delta$ small enough, Lemma~\ref{Lemma2} implies that each of these
measures satisfies $d_P(\mu',\mu) < \eps$. This establishes pro\-perty~(e).
In the case $\mu$ is already a periodic point of $\wt{h}$, of period $t$
say, then we can choose $p$ to be any multiple $kt$ of $t$, and so
we obtain uncountably many periodic points of $\wt{h}$ of period $kt$
in the $\eps$-neighborhood of $\mu$. This gives property~(b).

\smallskip
Finally, property (a) follows from the fact that a construction similar
to the one in Step~1 of Theorem~\ref{PeriodicC} can be made in the present
context, as was already mentioned above.
\end{proof}

%%%%%%%%%%%%%%%%%%%%%%%%%%%%%%%%%%%%%%%%%%%%%%%%%%%%%%%%%%%%%%%%%%%%%%%%%%%%%

\smallskip
\noindent Nilson C. Bernardes Jr.\\
{\it Departamento de Matem\'atica Aplicada, Instituto de Matem\'atica,
Universidade Federal do Rio de Janeiro, Caixa Postal 68530,
Rio de Janeiro, RJ, 21945-970, Brasil.}\\
{\it e-mail:} ncbernardesjr@gmail.com       %bernardes@im.ufrj.br

\smallskip
\noindent R\^omulo M. Vermersch\\
{\it Departamento de Tecnologias e Linguagens, Instituto Multidisciplinar,
Universidade Fede\-ral Rural do Rio de Janeiro, Av.\ Governador Roberto
Silveira s/n, Nova Igua\c cu, RJ, 26020-740, Brasil.}\\
{\it e-mail:} romulo.vermersch@gmail.com

\end{document}